\documentclass[twoside,leqno,11pt]{article}
\usepackage{revstat}

\usepackage{tikz}
\usetikzlibrary{calc}
\usetikzlibrary{decorations.pathreplacing}
\usepackage{amsmath}
\usepackage{amssymb}
\usepackage{amsfonts}
\usepackage{epsfig}
\usepackage{epstopdf} 
\usepackage{graphicx}
\usepackage{caption}
\usepackage{subcaption}
\usepackage{rotating}
\usepackage{bbm}
\usepackage{float}
\usepackage{multirow}
\usepackage{color}
\usepackage{mathtools}

\newcommand{\paren}[1]{\left(#1\right)}
\newcommand{\prect}[1]{\left[#1\right]}
\newcommand{\chav}[1]{\left\{#1\right\}}

\newcommand{\prectI}[1]{\left[ #1 \right.}
\newcommand{\prectF}[1]{\left. #1 \right]}

\newcommand{\parenF}[1]{\left. #1 \right)}

\begin{document}
\title{STUDY OF THE PARTICULAR SOLUTION OF A HAMILTON--JACOBI--BELLMAN EQUATION FOR A JUMP--DIFFUSION PROCESS
}
\renewcommand{\titleheading}
             {Particular Solution of a HJB equation}  
\author{\authoraddress{Cl{\'a}udia Nunes}
                      {Department of Mathematics and CEMAT, Instituto Superior T{\'e}cnico,\\
                       Av. Rovisco Pais, 1049--001 Lisboa, Portugal\\
                       (cnunes@math.tecnico.ulisboa.pt)}
\\
        \authoraddress{Rita Pimentel}
                      {RISE SICS V\"{a}ster\r{a}s AB,
                       \\
                       Stora Gatan 36, SE-722 12 V\"{a}ster\r{a}s, Sweden \\
                       (rita.duarte.pimentel@ri.se)}
\\
        \authoraddress{Ana Prior}
                      {Department of Mathematics, Instituto Superior de Engenharia de Lisboa,\\
                       R. Conselheiro Em\'idio Navarro 1, 1959-007 Lisboa, Portugal\\
                      (afpontes@adm.isel.pt)}
}
\renewcommand{\authorheading}
             {Nunes et al.}  

\maketitle

\begin{abstract}
We present an analytic solution of a differential-difference equation that appears when one solves an optimal stopping time problem with state process following a jump-diffusion process. 

This equation occurs in the context of real options and finance options, for instance, when one derives the optimal time to undertake a decision. Due to the jump process, the equation is not local in the boundary set. The solution that we present - which takes into account the geometry of the problem - is written in a backward form, and therefore its analysis (along with its implementation) is easy to follow.

\end{abstract}



\begin{ams}
37H10, 60G40. 
\end{ams}

\mainpaper  

\section{MOTIVATION}
In this paper we present the analytic solution of a differential--difference equation that appears when one solves optimal stopping problems with state process following a geometric Brownian motion with jumps driven by a Poisson process.  The main difficulty of this equation is consequence of the jump process, which in its turn implies that the equation is not local in one point (that, as we will see in the sequel, is the boundary between the continuation and stopping regions) -- see, for instance, Murto \cite{Murto}.

This characteristic is not universal, i.e., there are optimal stopping problems involving jump--diffusions processes for which the differential--difference equation does not exhibit this behaviour, and for which finding a closed form solution is easier. However, as we will show latter, this is not the case when the jumps of the diffusion may lead directly to the stopping region, across the boundary state.

The seminal works in finance options (like the classical work of Black and Scholes \cite{BlackScholes}, where for the first time an pricing formula was derived)  or real options (as the seminal book of Dixit and Pindyck \cite{Dixit:Pindyck:94}) assume that the sample path of the involved state process is continuous, with probability one. 
But a quick look at newspapers shows that nowadays investors need to take decisions facing uncertainty and the likelihood of financial crashes, which are the climax of the so--called log--periodic power law signatures associated with speculative bubbles (Johansen and Sornette \cite{JohansenSornette2010}). One example of this occurred in February 2015, when due to a cyber--attack, a high--frequency trading company started uncontrollably buying oil futures, causing a downward jump in the oil prices\footnote{https://www.businessinsider.com/investigation-into-hft-firm-for-using-an-algo-gone-wild-that-caused-oil-trading-mayhem-in-just-5-seconds-2010-8}. 
%

In this context, a crash is a significant drop in the total value of the market,  creating a situation wherein the majority of investors are trying to flee the market at the same time and consequently incurring massive losses. 
Indeed, in the presence of a crash investors likely take the decision to sell their assets. This is precisely what we mean by a jump leading to the stopping region (where, in this case, stopping means selling the asset), across the boundary. 
As the crash means that there is a significant drop, we borrow the probabilistic terminology and we call it a {\em jump} (in the above example, a downward jump).


These sudden changes in the state variable can also be found when one decides about (long term) investments in projects, often addressed in the context of real options. When it comes to real options, the temporal term is long, and therefore one may expect that during the life of an investment, unexpected events may occur, leading to a disrupt of the market. 

One example of a disrupt event is the introduction or abolishing of public subsidies. There are many economical sectors where subsidies play an important role, such as agriculture.
It is well established that agricultural pricing policies (taxes, subsidies) have a substantial
influence on farmer production decisions\footnote{http://www.pbl.nl/en/publications/the-impact-of-taxes-and-subsidies-on-crop-yields}. 
For example, USA has been supporting farming since early times. But after several decades, these incentive policies did not proved to be successful\footnote{https://grist.org/article/farm\_bill2/}.  In 2005 Bush administration decided to change the farm incentive policy, cutting in agricultural subsidies\footnote{https://www.agpolicy.org/weekpdf/258.pdf}. Evidently, this decision led to changes in private investment farming projects.

Another area where subsidies play an important role is the renewable energy (RE). 
In an effort to reach the ambitious targets of the EUÕs Strategic Energy Technology Plan
(SET--Plan), EU member states have implemented support mechanisms of various forms (e.g., price mechanisms, like carbon tax or permit trading schemes) intended to incentive and
accelerate adoption of RE technologies. These climate change policies have introduced a new
factor that has to be included into the investment decision and have become a major source of
uncertainty in energy strategy. The problem is that policies aimed at reducing emissions have frequently and
unexpectedly been changed for a number of reasons, for instance, change of governments, collapse of
the international cooperation for reducing GHG emissions, arrival of new information about
climate sensitivity, and fiscal pressure. 
In the last decade we have seen many works about the impact of wrong investment decisions. We refer, for instance, \cite{Boomsma2, Boomsma1, Hagspiel2018}.

In all the above examples, it is of the most crucial importance to assess the impact of the jumps in the decision, and, in particular, in case the jumps anticipate the optimal decision. 
The assumption about the jumps is crucial in the differential equation that holds in the continuation region. And when the jumps may lead directly to the stopping region, it is quite challenging to solve analitically this equation. And this is precisely our contribution with the paper.

The paper is organized as follows: in Section \ref{mw} we present the model, discuss the assumptions and present related work. In Section 3 the basic mathematics involved in the resolution of the problem is presented and in Section 4 a backwards procedure to obtain the solution of the problem is explained and illustrated with an example. The main results and the proofs are given in Section 5. Finally, Section 6 is devoted to the conclusions of the work.

\section{MODEL AND RELATED WORK}\label{mw}
The results derived in this paper can be used in different applications, although along the paper we describe and illustrate it in the context of real options. In particular, we consider the optimal moment to undertake an investment decision. 
We note that we could also change the setting in order to consider an exit decision, as the mathematics involved is similar, with simple changes.

This is an optimal stopping problem, formally defined as follows:  find $V(x)$ and $\tau^{\star} \in \mathcal{T}$ such that
\begin{equation}\label{eq:IntialProblem}
V(x) = \sup_{\tau \in \mathcal{T}} J^{\tau}(x) = J^{\tau^{\star}}(x), \quad x \in \R^+,
\end{equation}
with $\mathcal{T}$ being the set of all stopping times, and $J$ the discounted performance criterion,  given by
\begin{equation}\label{eq:J}
J^{\tau}(x) = \mathbb{E}^{x} \prect{e^{-r \tau} {g(X({\tau}))} \chi_{\chav{\tau < + \infty}}},
\end{equation}
where $r>0$ states for the discount factor, $\chi_{\chav{A}}$ represents the indicator function on set $A$, and $\{X(t),t>0\}$ represents the stochastic process that accounts for the uncertainty of the investment. For example, in finance options $X(t)$ may represent the underlying stock price at time $t$, whereas in real options may represent the price of a product at time $t$. In \eqref{eq:J}, the function $g$ is the so--called {\em running function}, which accounts for the return of the investment.

In the current work, we assume that the  state process dynamics, $\{X(t), t>0\}$,  is a one--dimensional jump--diffusion. In fact, it is the solution of the following stochastic differential equation
\begin{equation}\label{eq:geralC}
\frac{d X(t)}{X(t^-)} = \mu dt + \sigma  d W(t) + \kappa dN(t),
\end{equation}
with initial value $X(0) = x > 0$, where $\{W(t), t>0\}$ is a standard one--dimensional Brownian motion, and $\{N(t),t>0\}$ 
is a time--homogeneous Poisson process, with intensity $\lambda>0$. Moreover, $\mu$ is the drift of the process, $\sigma > 0$ is the volatility and $\kappa$ is the multiplicative factor, in case a jump occurs. 
The notation $X(t^-)$ means that whenever there is a jump, the value of the process before the jump is used on the left--hand side of the formula.
%
The solution of \eqref{eq:geralC} is given by
\begin{equation}\label{eq:gbm+J}
X (t) = x \; e^{\prect{(\mu-\frac{\sigma^2}{2}-\lambda \kappa) t +\sigma W(t)}} \left( 1+ \kappa\right)^{N(t)}.
\end{equation}

The interest of real options literature in problems involving jump--diffusion processes 
is not new. We refer to  Kou \cite{Kou} for a survey on jump-diffusion models for finance engineering. In the area of real options, there has been an increasing interest about jump--diffusion processes in the context of technology adoption. 
We refer, for instance, to Balcer and Lippman \cite{Balcer84}, Farzin et al. \cite{FHK1998}, Huisman \cite{Huisman} and Hagspiel et al. \cite{Hagspiel2015}. 
These papers study the best time to invest in a new and more efficient technology, assuming that the process that describes the innovative process of a technology development can be modeled by a Poisson process, and therefore it is a pure jump--process 
(which, in the case of  Hagspiel et al.  \cite{Hagspiel2015}, is non--homogeneous with time--dependent intensity). 

Furthermore, Kown \cite{Kwon} and Hagspiel et al. \cite{Hagspiel2016} consider a combination of a continuous process with a jump--process, but they do not consider a sequence of innovations arriving over time. Instead, they assume a one--single innovation opportunity, with other involved options (like the option to exit the market). Kwon \cite{Kwon} work is generalized in Hagspiel et al. \cite{Hagspiel2016}, by considering capacity optimization, and by Matom{\"a}ki \cite{matomaki2013two}, considering different stochastic processes representing the profit uncertainty.

In another context, Couto et al. \cite{CoutoNunesPimentel} and  Nunes and Pimentel \cite{NunesPimentel} consider the investment problem in a high--speed railway service,  assuming that both the demand and the investment cost are modeled by jump--diffusion processes. Although these papers start by assuming two sources of uncertainty, they end up with the study of a one-dimensional problem. This happens because they assume that the value of the firm is homogeneous, and therefore it is possible to consider a change of variables that will turn the 2--dimensional problem in a 1--dimensional one.  Murto \cite{Murto} also consider two stochastic processes, in order to model technological and revenue uncertainties, motivated by wind power investment. He assumes that 
the investment cost depends on the technological progress, driven by a pure Poisson process, whereas the price of the output is a geometric Brownian motion. As the value of the project is  homogeneous, the same kind of approach as in Nunes and Pimentel \cite{NunesPimentel} is proposed. 

Motivated by these references, we assume that the jumps are multiplicative, its magnitude is constant and they are such that the process may enter the stopping region with a jump. In case we are dealing with an investment problem, the {{\em stopping region} is given by $[x^\star, + \infty )$,  
where $x^\star$ is the exercise threshold. 
Then, the stopping region may be attained due to the occurrence of a jump if and only if $\kappa>0$. 
Moreover, in this case the stopping region can be reached in two different ways:
\begin{itemize}
\item[(i)] Either due a continuous change, caused by the diffusion part, where the state process {\em hits} the boundary threshold; 
\item[(ii)] Or due to the occurrence of a jump, where the state process {\em crosses} the boundary threshold.
\end{itemize}
If $\kappa<0$, every time there is a jump, the state process would fall back\footnote{In case we would be dealing with exit options, it would be the other way: the optimal decision to exit may be anticipated in case the jumps are negative.} and therefore the stopping region may only be reached through the diffusion part.

This is precisely the main difference of our model with respect to the above mentioned references -- with the exception of Merton \cite{Merton:76} and Murto \cite{Murto}. The other references  deal with situations where either there is just the jump process (for which it is possible to solve the corresponding difference equation, as there is no differential part) -- as it is the case of Huisman  \cite{Huisman} -- or the process is a jump--diffusion but the jumps always lead to the continuation region -- as it is the case of Nunes and Pimentel \cite{NunesPimentel}. 

Our work is closely related with Merton \cite{Merton:76}, in the sense that he considers a model (to price American options) similar to the one presented here. He also assumes multiplicative jumps, but his model is more general, in the sense that the magnitude of the jumps is random. However, he does not provide a closed--form expression for the price but instead a computational efficient formula (equation (14) of Merton \cite{Merton:76}). 

Finally, we assume that the stopping time is not bounded, meaning that this is an infinite time--horizon problem.  The infinite--maturity hypothesis helps to reduce the dimensionality of the problem
by removing its dependence on time, therefore concentrating only on stationary solutions.  We follow, for instance, Dixit and Pindyke \cite{Dixit:Pindyck:94}, where the decisions to start,
abandon, reactive and mothball a given project are reduced to the solution of a systems of linear equations, with analytic solutions.

In the following section we present in more detail the problem and the equation that we are able to provide a closed expression solution.

\section{DIFFERENTIAL--DIFFERENCE EQUATION}

In this section we present the basic mathematics involved in the resolution of the stopping problem (\ref{eq:IntialProblem}). We refer to \cite{Applebaum:2004, Bass:76, Oksendal2009, Sennewald:05}, for instance, for proofs, derivations and further comments on the results that we will use along this section.

The infinitesimal generator associated to \eqref{eq:gbm+J} corresponds to adding a drift and diffusion term to a jump operator, i.e.
\begin{equation*}
\mathcal{L} v(x) = \frac{\sigma^2}{2} x^2 v^{\prime \prime}(x) + (\mu - \lambda \kappa) x v^\prime (x) + \lambda \paren{ v(x (1+\kappa) ) - v(x)}, 
\end{equation*} 
for $v \in C^1$ 
and $x \in \R^+$. 
Furthermore, the Hamilton--Jacobian--Bellman (HJB) equation of problem \eqref{eq:IntialProblem} takes the form of the variational inequality
\begin{equation*}\label{HJB}
\min \chav{r V(x) - \mathcal{L}V(x) , V(x) - g(x)}=0.
\end{equation*}
%
%
%
%
%

From Propositions 3.3 and 3.4 of {\O}ksendal and Sulem--Bialobroda \cite{Oksendal2009}, we know that the following set is contained in the continuation region
$${\cal U}=\{x \in \R^+: r g(x) - \mathcal{L} g(x) = 0\}.$$
This means that, for each value of $x \in {\cal U}$, the value of continuing  is larger than the value of stopping and therefore the decision is postponed.
Taking into account the $\cal U$ set and the usual assumption that $r > \mu$, we can guarantee that it is not optimal to stop for lower values of $x$. 
Then, as it was presented before, the stopping region is given by $[x^\star,+\infty)$, whereas $(0,x^\star)$ is the continuation region.
Therefore, the optimal stopping time, $\tau^\star$, is such that $\tau^\star=\inf\{t>0 : X_t \geq x^\star\}$, where $x^\star$ is a threshold that need to be found.

In the stopping region $V$ is equal to $g$, i.e. $V(x)=g(x)$ for $x \geq x^\star$. 
 On the other hand, in the continuation region $V$ verifies the following equation
\begin{equation}\label{OriginalEq}
x^2 V^{\prime \prime}(x)+ a \; x V^\prime(x) + b \; V(x) - c \; V(x(1+k)) = 0,
\end{equation}
where $a=\frac{2 (\mu-\lambda k)}{\sigma^2}$, $b=- \frac{2(r+\lambda)}{\sigma^2}$ and $c=- \frac{2 \lambda}{\sigma^2}$.
This is called in the literature {\em mixed partial differential--difference equation} (see Merton \cite{Merton:76}), and it is known to be difficult to solve, specially when one impose boundary conditions. We refer to Johansen and Zervos \cite{Johnson2007solution}, where analytic and
probabilistic properties for the solution of an equation similar to \eqref{OriginalEq} are provided (but with other specificities, notably in terms of the coefficients of the terms involved in the equation).

We note that Equation (\ref{OriginalEq}) is a particular case of a delay differential equation. The interest in this class of equations is growing in all scientific areas, especially in control engineering. Delay differential equations are usally difficult to solve, and in many cases one needs to resort to numerical schemes. Thus, this class of equations has been subject of intense study, as Arino et al. \cite{Arino} shows. 
As we will explain in the rest of the paper, in the scope in which we are working, one can provide an analytical solution for Equation \eqref{OriginalEq}.


%
%
%
%

\section{BACKWARDS ANALYSIS}
%
Given the geometry of the solution of the optimal stopping time problem \eqref{eq:IntialProblem}, we present a backwards procedure to solve \eqref{OriginalEq}. 

If we consider $x \in \parenF{\prectI{\frac{x^{\star}}{1+ \kappa}, x^{\star}}}$ 
then $x (1 + \kappa) \in \prectI{x^{\star} , + \infty )}$, meaning that $V(x(1+\kappa))=g(x(1+\kappa))$. Thus, in this case Equation \eqref{OriginalEq} can be re-written as
\begin{equation}\label{OSjumps2}
x^2 V^{\prime \prime}(x)+ a \; x V^\prime(x) + b \; V(x) = c \; g(x(1+\kappa)).
\end{equation}
Equation \eqref{OSjumps2} is a non--homogeneous second order ODE. 
The correspondent homogeneous equation,
\begin{equation*}\label{eqEC}
x^2 V^{\prime \prime}(x)+ a \; x V^\prime(x) + b \; V(x) = 0,
\end{equation*}
is called second order Euler--Cauchy equation.
This equation appears in a wide range of problems, such as sorting and searching algorithms (see Chern et al. \cite{Chern:2002:ATC:606234.606243}) or physics and engineering applications (see Chen and Wang \cite{ChenWang}). 
Its solution, hereby denoted by $V_h$, is well--known (see, for instance, \cite{2008elementary, 2006elementary}) and strongly depends on the roots of the characteristic polynomial 
\begin{equation}\label{discPoly}
Q(\beta) = \beta \paren{\beta - 1} + a \beta + b.
\end{equation}
%
In our case, given that $b<0$, $Q$ has two distinct real roots\footnote{$Q(\beta) = 0 \Leftrightarrow \beta = \frac{1}{2} \prect{ 1 - a \pm \sqrt{(1-a)^2 - 4 b}}$.}, say $\beta_1>0$ and $\beta_2<0$. Then, it follows that 
\begin{equation}\label{eq:homSol}
V_h(x) = \delta_1 x^{\beta_1} + \delta_2 x^{\beta_2}.
\end{equation}
The (general) solution of \eqref{OSjumps2}, 
which we denote by $V_1$, is given by $V_h$ plus a particular solution, which we denote by $V_p^{1}:=f_g^{1}$. 
Note that the superscript in $V_p^1$ and $f_g^1$ represents how many jumps we are away from the stopping region\footnote{We use this type of notation for all particular solutions.}. Moreover, the bottom index in $f_g^{1}$ emphasizes that this function depends explicitly on $g$.
For $x \in \parenF{\prectI{\frac{x^{\star}}{1+ \kappa}, x^{\star}}}$, we have
\begin{equation}
\label{V1}
V(x):= V_1(x) = V_h(x) + V_p^{1}(x) = \delta_1 x^{\beta_1} + \delta_2 x^{\beta_2} + f_g^{1}(x), 
\end{equation}
where $\delta_1$ and $\delta_2$ are parameters that need to be determined (see Figure \ref{fig:V1} for an illustration). 

\begin{center}
\begin{figure}[h!]
\centering
\begin{tikzpicture}
\draw[thick,->,ultra thick] (-5,0) -- (5,0) node[anchor=north west] {$x$};
\node[align=center] at (4,0.5) {$\mathcal{S}$};
\node[align=center] at (-3,0.5) {$\mathcal{C}$};
\node[align=center] at (1,-0.4) {$\frac{x^\star}{1+\kappa}$};
\node[align=center] at (2.5,-0.3) {$x^\star$};
\node at (1,0) [circle,fill=black,scale=0.4] {};
\node at (2.5,0) [circle,fill=black,scale=0.5] {};
\path[draw,ultra thick] (-5,-0.2)--(-5,0.2);
\draw[decorate,decoration={brace,amplitude=5pt,mirror},xshift=0pt,yshift=-20]	(2.6,0) -- (4.8,0);
\draw (3.7,-1.2) node {\footnotesize $g(x)$};
\draw[decorate,decoration={brace,amplitude=5pt,mirror},xshift=0pt,yshift=-20]	(1,0) -- (2.4,0);
\draw (3.7,-1.2) node {\footnotesize $g(x)$};
\draw (1.7,-1.2) node {\footnotesize $V_1(x)$};
\end{tikzpicture}
\caption{Representation of $V$ in the last interval before stopping.}
\label{fig:V1}
\end{figure}
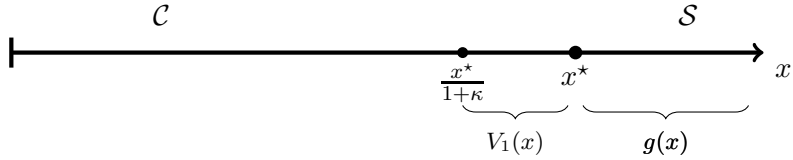
\end{center}

Next we derive the value of $V$ when we are two jumps away from the stopping region. Following the same notation, we denote this function by $V_2$, defined for $x \in \left[\frac{x^\star}{(1+\kappa)^2}, \frac{x^\star}{1+\kappa}\right)$.
In this case $x (1+\kappa) \in \parenF{\prectI{\frac{x^{\star}}{1+ \kappa}, x^{\star}}}$, so $V(x(1+\kappa)) = V_1 (x(1+\kappa))$. This means that  \eqref{OriginalEq} can be re--written as follows
\begin{equation*}
x^2 V^{\prime \prime}(x)+ a \; x V^\prime(x) + b \; V(x) = c \; V_1(x(1+\kappa)).
\end{equation*}
The homogeneous part of the previous equation is the same as before, and thus the  solution is provided in \eqref{eq:homSol}. We just need to take into account the particular solution, which we denote by $V_p^{2}$. 
This particular solution depends on $V_p^{1}$ (and thus depends on $g$) but also depends on $V_h$ (then also depends on the roots of $Q$, $\beta_1$ and $\beta_2$), as $V_1$ is given by \eqref{V1}. 
Therefore, both the homogeneous and the particular solution for this case {\em share} the powers $\beta_1$ and $\beta_2$. Using Theorem 3.5 of Sabuwala  and De Leon \cite{Doreen}, we end up with the following particular solution
\begin{equation}
V_p^{2}(x) =  \eta_1^{2} \ln x \; x^{\beta_1} + \eta_2^{2} \ln x \; x^{\beta_2} + f_g^{2}(x). 
\end{equation}
We write $f^{2}_g$ to denote the part of the solution that depends strictly on $g$ (following the same reasoning as for $f_g^{1}$), 
 whereas $\eta_1^{2}$ and $\eta_2^{2}$ depend on the parameters from the homogeneous solution.
So, for $x \in \left[\frac{x^\star}{(1+\kappa)^2}, \frac{x^\star}{1+\kappa}\right)$ 
 (see Figure \ref{fig:V2} for an illustration)
 , we have 
 \begin{equation}
 \label{V2}
 V(x):= V_2(x)= \delta_1 x^{\beta_1} + \delta_2 x^{\beta_2}+\eta_1^{2} \ln x \; x^{\beta_1} + \eta_2^{2} \ln x \; x^{\beta_2} + f_g^{2}(x).
 \end{equation}

\begin{center}
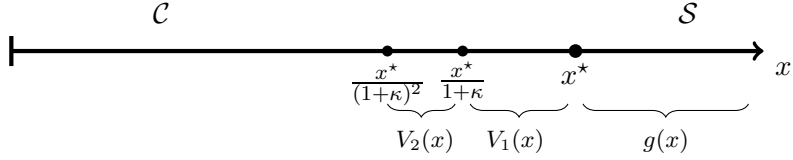
\begin{figure}[h!]
\centering
\begin{tikzpicture}
\draw[thick,->,ultra thick] (-5,0) -- (5,0) node[anchor=north west] {$x$};
\node[align=center] at (4,0.5) {$\mathcal{S}$};
\node[align=center] at (-3,0.5) {$\mathcal{C}$};
\node[align=center] at (0,-0.45) {$\frac{x^\star}{(1+\kappa)^2}$};
\node[align=center] at (1,-0.4) {$\frac{x^\star}{1+\kappa}$};
\node[align=center] at (2.5,-0.3) {$x^\star$};
\node at (0,0) [circle,fill=black,scale=0.4] {};
\node at (1,0) [circle,fill=black,scale=0.4] {};
\node at (2.5,0) [circle,fill=black,scale=0.5] {};
\path[draw,ultra thick] (-5,-0.2)--(-5,0.2);
\draw[decorate,decoration={brace,amplitude=5pt,mirror},xshift=0pt,yshift=-20]	(2.6,0) -- (4.8,0);
\draw[decorate,decoration={brace,amplitude=5pt,mirror},xshift=0pt,yshift=-20]	(1.1,0) -- (2.4,0);
\draw[decorate,decoration={brace,amplitude=5pt,mirror},xshift=0pt,yshift=-20]	(0,0) -- (0.9,0);
\draw (3.7,-1.2) node {\footnotesize $g(x)$};
\draw (1.7,-1.2) node {\footnotesize $V_1(x)$};
\draw (0.5,-1.2) node {\footnotesize $V_2(x)$};
\end{tikzpicture}
\caption{Representation of $V$ in the last two intervals before stopping.}\label{fig:V2}
\end{figure}
\end{center}

Proceeding with one step back, we determine the value of $V$ when we are three jumps aways from the stopping region, which we call $V_3$.
When $x \in \left[\frac{x^\star}{(1+\kappa)^3}, \frac{x^\star}{(1+\kappa)^2}\right)$, then $x (1+\kappa) \in \left[\frac{x^\star}{(1+\kappa)^2}, \frac{x^\star}{1+\kappa}\right)$ and $V(x(1+\kappa)) = V_2 (x(1+\kappa))$. Then, Equation \eqref{OriginalEq} is re-written as
\begin{equation}
\label{step3}
x^2 V^{\prime \prime}(x)+ a \; x V^\prime(x) + b \; V(x) = c \; V_2(x(1+\kappa)).
\end{equation}
As before, the homogeneous equation is the same and therefore $V_h$ is part of the solution of this equation. Once more, the problem is reduced to the derivation of a particular solution, which is not trivial, as the function $V_2$ involves polynomials of power $\beta_1$ and $\beta_2$ multiplied by a logarithm (see Equation \eqref{V2}).
After some calculations, one may find that the particular solution of \eqref{step3} is of the following form
$$V^{3}_p(x) = \eta_1^{3} \ln x \;  x^{\beta_1} +  \eta_2^{3} \ln x \;  x^{\beta_2} + \eta_3^{3} \paren{\ln x}^2  x^{\beta_1} + \eta_4^{3} \paren{\ln x}^2  x^{\beta_2} + f_g^{3}(x).$$
Also here $f_g^3$ stands for the part of the solution that depends strictly on $g$ whereas $\eta_1^{3}, \eta_2^{3}, \eta_3^{3}$ and $\eta_4^{3}$ depend on the parameters from the homogeneous solution. 
As previously, for $x \in \left[\frac{x^\star}{(1+\kappa)^3}, \frac{x^\star}{(1+\kappa)^2}\right)$, we have
\begin{eqnarray*}
V(x) := V_{3}(x) &=&  \delta_1 x^{\beta_1} + \delta_2 x^{\beta_2} + \eta_1^{3} \ln x \;  x^{\beta_1} +  \eta_2^{3} \ln x \;  x^{\beta_2} \\
&& + \eta_3^{3} \paren{\ln x}^2  x^{\beta_1} + \eta_4^{3} \paren{\ln x}^2  x^{\beta_2} + f_g^{3}(x).
\end{eqnarray*}

A similar reasoning applies for other intervals of $x$.
When we are $i$ (with $i \in \N$) jumps away from the stopping region, we have $\frac{x^{\star}}{(1+\kappa)^i} \leq x < \frac{x^{\star}}{(1+\kappa)^{i-1}}$ and $V$ is represented by $V_i$, which may be obtained using a similar reasoning as the one that we have described for $V_1, V_2$ and $V_3$. 
Indeed, $V$ is a piecewise function, given by
\begin{equation}\label{Vdefinition}
V(x) = 
\begin{cases}
V_i (x) & \text{if} \quad \frac{x^{\star}}{(1+\kappa)^i} \leq x < \frac{x^{\star}}{(1+\kappa)^{i-1}}\\
g(x) & \text{if} \quad x \geq x^{\star}
\end{cases},
\end{equation}
where
$$V_i(x)=\delta_1 x^{\beta_1}+\delta_2 x^{\beta_2}+V_p^{i}(x)$$
with
\begin{eqnarray}
&& V_p^1(x) = f_g^1(x) \quad \text{and} \quad \notag \\
&& V_p^i(x) = \sum_{j = 1}^{i-1} \prect{\eta^i_{2 j -1} \;  x^{\beta_1} + \eta^i_{2 j} \;  x^{\beta_2}} \paren{\ln x}^j + f_g^i(x), \;\text{for} \; i \in \N \setminus \chav{1}. \label{eq:Vpi}
\end{eqnarray}

Clearly, one needs to find functions that are solutions of certain differential equations, that depend intrinsically on the function $g$, considered in the definition of the problem.
In order to illustrate the above derivations, we present next the  calculations for a particular case, that we find often in the literature of real options.

\begin{example}\label{Example1}
Let us consider a firm that want to decide about investing in a new product. 
The uncertainty is given by the demand for such a product, where $X(t)$ represents this demand at time $t$.
Let us assume that $\chav{X(t) : t>0}$ follows a jump-diffusion process, as the one described in \eqref{eq:gbm+J}.
%
 Moreover, let us assume that after investment, the firm's profit is given by
$$g(x) =\rho x^{\theta} - I,$$
with $\theta$ denoting the elasticity parameter with respect to the demand process. This is an usual assumption for the profit of firms, and it is called in the literature as $g$ being an iso--elastic demand function (see Nunes and Pimentel \cite{NunesPimentel}).
For simplicity and without loss of generality, we assume that  $\theta$ is not a root of the characteristic polynomial $Q$.

As we explained before, the general solution of the optimal stopping time problem \eqref{eq:IntialProblem} is given by \eqref{Vdefinition}.
In this example we want to present the particular  solutions $V_p^{i}$, for $i=1,2,3$. Doing some calculations, we can get the expressions of the parameters for each function.
\begin{align*}
V_p^{1}(x) = &  \xi_1^{1} x^{\theta} + \xi_2^{1}, \quad \text{with}  \quad \xi_1^{1}=\frac{c \rho (1 + \kappa)^\theta}{Q(\theta)},\;\xi_2^{1}=-\frac{c I}{b}.
\end{align*}
\begin{align*}
V_p^{2}(x) = & \eta_1^{2} \ln x \; x^{\beta_1} + \eta_2^{2} \ln x \; x^{\beta_2} +  \xi_1^{2} x^{\theta} + \xi_2^{2}  \quad \text{with} \quad \\ 
& \eta_1^{2}=\delta_1 \frac{ c  (1+\kappa)^{\beta_1}}{Q^\prime(\beta_1)}, \; \eta_2^{2}= \delta_2 \frac{c (1+\kappa)^{\beta_2}}{Q^\prime (\beta_2)}, \;\\ 
& \xi_1^{2}= \rho \prect{\frac{c (1+\kappa)^\theta}{Q(\theta)}}^2, \; 
 \xi_2^{2}= - \paren{\frac{c}{b}}^2 I.
\end{align*}
\begin{align*}
V_p^{3}(x) = & 
\eta_1^{3}  \ln x \;  x^{\beta_1} + 
  \eta_2^{3}  \ln x \;  x^{\beta_2} + 
\eta_3^{3} \paren{\ln x}^2 \;  x^{\beta_1} + \eta_4^{3}  \paren{\ln x}^2 \;  x^{\beta_2} + \xi_1^{3} x^\theta + \xi_2^{3},\\
&\text{with} 
 \; \\
& \eta_1^{3}= \delta_1 \frac{ c  (1+\kappa)^{\beta_1}}{Q^\prime(\beta_1)} \prect{1 + \frac{c (1+ \kappa)^{\beta_1}}{Q^\prime(\beta_1)} \paren{\ln (1+\kappa) - \frac{1}{Q^\prime(\beta_1)}} }
\\
& \eta_2^{3}= \delta_2 \frac{ c  (1+\kappa)^{\beta_2}}{Q^\prime(\beta_2)} \prect{1 + \frac{c (1+ \kappa)^{\beta_2}}{Q^\prime(\beta_2)} \paren{\ln (1+\kappa) - \frac{1}{Q^\prime(\beta_2)}} }
\\
& \eta_3^{3} = \frac{\delta_1}{2} \prect{\frac{c (1+ \kappa)^{\beta_1}}{Q^\prime (\beta_1)}}^2, \; \eta_4^{3} = \frac{\delta_2}{2} \prect{\frac{c (1+ \kappa)^{\beta_2}}{Q^\prime (\beta_2)}}^2 \\
& \xi_1^{3}=\rho \prect{\frac{c (1+\kappa)^\theta}{Q(\theta)}}^3, \; \xi_2^{3}=- \paren{\frac{c}{b}}^3 I.
\end{align*}
%
\end{example}

Note that the results presented in the previous example follow from straightforward calculations. However, if one want to show the results for each $i\in\mathbb{N}$, then it will lead to tedious and not very fast calculations. 
In addition, we emphasize that for all the cases, one needs to find the particular solution of an ODE of order two, which belongs to the following class of ODEs:  
\begin{equation}\label{TheEq}
x^2 y''(x) + a x y'(x) + b y(x) = A x^\alpha (\ln x)^n,
\end{equation}
with $x > 0$, $a, b \in \R$, $\alpha, A \in \R \setminus \chav{0}$ and $n \in \N_0$.

In the next section we present results that lead to a general way to compute the particular solution of Equation \eqref{TheEq}, which we denote by $y_p$.

\section{MAIN RESULTS}

We want to find a particular solution to Equation \eqref{TheEq}. 
The type of solution is understandable from the special case presented in the Example \ref{Example1}. However, a systematic way to obtain all the coefficients is not so easy to develop. 

We start deriving a recursive expression for the particular solution of  \eqref{TheEq}. 
Later, using this result, we will be able to present explicit expressions for the involved coefficients.

\begin{theorem}[recursive]\label{Teo1}
Consider the second order ODE presented in \eqref{TheEq}, with
the corresponding characteristic polynomial $Q$ given by \eqref{discPoly}. Then the following cases occur:

\begin{itemize}
\item If $\alpha$ is not a root of $Q$, the particular solution of \eqref{TheEq} is $$y_p(x) = x^\alpha \sum_{i=0}^{n} c_i \paren{\ln x}^i,$$ where $c_n=\frac{A}{Q(\alpha)}$, $c_{n-1}=- n A \frac{Q'(\alpha)}{Q(\alpha)^2}$ and $c_i=-\frac{i+1}{Q(\alpha)} \prect{Q'(\alpha)  c_{i+1} + (i+2) c_{i+2}}$ for $i=0, 1, 2, ..., n-2$.
\item If $\alpha$ is a root of $Q$ with multiplicity one, the particular solution of \eqref{TheEq} is $$y_p(x) = x^\alpha \sum_{i=0}^{n} c_i \paren{\ln x}^{i+1},$$ where
$c_n=\frac{A}{(n+1) Q'(\alpha)}$ and $ c_i = - \frac{i+2}{Q'(\alpha)} c_{i+1}$, for $i=0,1,2, ..., n-1$.
\item If $\alpha$ is a root of $Q$ with multiplicity two, the particular solution of \eqref{TheEq} is $$y_p(x) = x^\alpha c_n \paren{\ln x}^{n+2},$$ where
$c_n=\frac{A}{(n+1)(n+2)}$.
\end{itemize}
\end{theorem}

\begin{proof}
We start by proposing that the particular solution of Equation \eqref{TheEq} is of the form $y_p(x)=x^\alpha P(x)$. Calculating first and second derivatives, we obtain 
\begin{eqnarray*}
y_p^\prime(x) &=& x^{\alpha-1} \prect{ x P^\prime(x)+\alpha P(x)} \\
y_p^{\prime \prime}(x) &=& x^{\alpha-2}\prect{ x^2 P^{\prime \prime}(x) + 2 \alpha x P^\prime(x) + \alpha (\alpha-1) P(x)},
\end{eqnarray*}
from where
\begin{eqnarray*}
x^2 y_p^{\prime \prime}(x) + a x y_p^{\prime}(x) + b y_p(x) &=& x^\alpha  \prect{ x^2 P^{\prime \prime}(x) + \paren{Q^\prime(\alpha)+1} x P^\prime(x) + Q(\alpha) P(x)}.
\end{eqnarray*}
Thus $P(x)$ is such that 
\begin{equation}\label{cond1}
x^2 P^{\prime \prime}(x) + \paren{Q^\prime(\alpha)+1} x P^\prime(x) + Q(\alpha) P(x) = A (\ln x)^n.
\end{equation}
Taking into account whether $Q(\alpha)$ is null or not, we end up with different cases, described hereafter.
\begin{enumerate}
\item If $\alpha$ is not a root of $Q$, then $P(x)=\sum_{i=0}^{n} c_i \paren{\ln x}^i$, as we prove next. For that, we compute the first and second derivatives:
\begin{eqnarray*}
P^\prime(x) &=& \frac{1}{x} \sum_{i=1}^{n} i c_i \paren{\ln x}^{i-1} \\
P^{\prime \prime}(x) &=& \frac{1}{x^2}\prect{\sum_{i=2}^{n}{i(i-1)c_i (\ln x)^{i-2}} - \sum_{i=1}^{n}{i c_i (\ln x)^{i-1}}}.
\end{eqnarray*}
Thus, $x^2 P^{\prime \prime}(x) + \paren{Q^\prime(\alpha)+1} x P^\prime(x) + Q(\alpha) P(x)$ is given by
\begin{eqnarray*}
&&\sum_{i=0}^{n-2}{\prect{(i+2)(i+1)c_{i+2}+Q^\prime(\alpha)(i+1)c_{i+1}+Q(\alpha)c_i}(\ln x)^i} \\
&&+ \prect{Q^\prime(\alpha) n c_n + Q(\alpha) c_{n-1}} (\ln x)^{n-1} + Q(\alpha) c_n (\ln x)^n.
\end{eqnarray*}
Thus  \eqref{cond1} holds if $ Q(\alpha) c_n = A$, $Q^\prime(\alpha) n c_n + Q(\alpha) c_{n-1} = 0$ and $(i+2)(i+1)c_{i+2}+Q^\prime(\alpha)(i+1)c_{i+1}+Q(\alpha)c_i = 0 , \; \text{for} \; i=0, 1,..., n-2$, which leads to the result.
%
%
\item If $\alpha$ is a root of $Q$ with multiplicity one, then $P(x)=\sum_{i=0}^{n} c_i \paren{\ln x}^{i+1}$. In fact, calculating first and second derivatives, we obtain 
\begin{eqnarray*}
P^\prime(x) &=& \frac{1}{x} \sum_{i=0}^{n} (i+1) c_i \paren{\ln x}^{i} \\
P^{\prime \prime}(x) &=& \frac{1}{x^2}\prect{\sum_{i=1}^{n}{(i+1)i c_i (\ln x)^{i-1}} - \sum_{i=0}^{n}{(i+1) c_i (\ln x)^{i}}}.
\end{eqnarray*}
Given that $Q(\alpha)=0$, then $x^2 P^{\prime \prime}(x) + \paren{Q^\prime(\alpha)+1} t P^\prime(x) + Q(\alpha) P(x)$ is given by $\sum_{i=0}^{n-1} \prect{(i+2)c_{i+1}+Q^\prime(\alpha)c_i} (\ln x)^i + Q^\prime(\alpha) (n+1) c_n (\ln x)^n$.

Assuming that $\alpha$ has multiplicity one we have $Q'(\alpha) \neq 0$. Thus, in order to have \eqref{cond1}, we need to set that $Q^\prime(\alpha) (n+1) c_n = A$ and $(i+2)c_{i+1}+Q'(\alpha)c_i = 0 , \; \text{for} \; i=0, 1,..., n-1$, and the result follows. 

%
%
\item If $\alpha$ is a root of $Q$ with multiplicity two, then $P(x)= c_n \paren{\ln x}^{n+2}$ as
\begin{eqnarray*}
P^\prime(x) &=& \frac{1}{x} c_n (n+2) \paren{\ln x}^{n+1} \\
P^{\prime \prime}(x) &=& \frac{1}{x^2} c_n (n+2)\prect{(n+1) (\ln x)^n - (\ln x)^{n+1}}.
\end{eqnarray*}
Since $Q(\alpha)=0$ and $Q^\prime(\alpha)=0$, then $x^2 P^{\prime \prime}(x) + \paren{Q^\prime(\alpha)+1} t P^\prime(x) + Q(\alpha) P(x)$ is given by $c_n (n+2) (n+1) (\ln x)^n$.
Finally, in order to have \eqref{cond1} we conclude that $c_n=\frac{A}{(n+1) (n+2)}$.
\end{enumerate}
\end{proof}

This theorem is useful in two ways: first it provides a  way to compute (recursively) the particular solution of the differential equation \eqref{TheEq}. Second, it provides the tool to derive explicit expressions for the involved coefficients. In the following theorem we present such result. 

\begin{theorem}[non-recursive]\label{Teo2}
Consider the second order ODE presented in \eqref{TheEq}, with the corresponding characteristic polynomial $Q$ given by (\ref{discPoly}).
%
\begin{itemize}
\item If $\alpha$ is not a root of $Q$, the particular solution of \eqref{TheEq} is given by $y_p(x) = x^\alpha \sum_{i=0}^{n} c_i \paren{\ln x}^i$, with
\begin{eqnarray}
c_i &=&  (-1)^{n-i} \; \frac{n!}{i!} \; \frac{A}{Q(\alpha)^{n-i+1}} \; \times \label{ciNoRoot} \\
&& \sum\limits_{\substack{j=0 \\ j \in \N_0}}^{\frac{n-i}{2}}{(-1)^j \; \binom{n-i-j}{j} \; Q'(\alpha)^{n-i-2j} \; Q(\alpha)^j}, \notag
\end{eqnarray}
for $i=0, 1, 2,.., n$, where $\binom{k}{r} = \frac{k!}{r! (k-r)!}$, with $k \geq r \geq 0$.

\item If $\alpha$ is a root of $Q$ with multiplicity one, the particular solution of \eqref{TheEq} is $y_p(x) = x^\alpha \sum_{i=0}^{n} c_i \paren{\ln x}^{i+1}$, with
\begin{equation}\label{ci1Root}
c_i = (-1)^{n-i} \; \frac{n!}{(i+1)!} \; \frac{A}{Q'(\alpha)^{n-i+1}}, \; \text{for} \; i=0, 1, 2,.., n.
\end{equation}

\item If $\alpha$ is a root of $Q$ with multiplicity two, the particular solution of \eqref{TheEq} is $y_p(x) = x^\alpha c_n \paren{\ln x}^{n+2}$, with $c_n=\frac{A}{(n+1)(n+2)}$.
\end{itemize}
\end{theorem}

\begin{proof}
The last case coincides with the one presented in Theorem \ref{Teo1}. For the other two cases, we use backwards mathematical induction to prove it, taking advantage of the recursive solutions presented in Theorem \ref{Teo1}.

\begin{enumerate}
\item If $\alpha$ is not a root of $Q$, we already know that, the particular solution is of the form $y_p(x) = x^\alpha \sum_{i=0}^{n} c_i \paren{\ln x}^i$, where $c_n=\frac{A}{Q(\alpha)}$, $c_{n-1}= - n A \frac{Q'(\alpha)}{Q(\alpha)^2}$ and $c_i=-\frac{i+1}{Q(\alpha)} \prect{Q'(\alpha)  c_{i+1} + (i+2) c_{i+2}}$ for $i=0, 1, 2, ..., n-2$. We want to prove that, for $i=0, 1, 2,.., n$, the coefficients  $c_i$ can be written in the general form presented in \eqref{ciNoRoot}.

Using backwards mathematical induction we have two base cases to be verified, $c_n$ and $c_{n-1}$, which we know from Theorem \ref{Teo1} that are $\frac{A}{Q(\alpha)}$ and $- n A \frac{Q'(\alpha)}{Q(\alpha)^2}$, respectively. Taking into account \eqref{ciNoRoot}, we have
\begin{eqnarray*}
c_n &=& (-1)^{0} \; \frac{n!}{n!} \; \frac{A}{Q(\alpha)} \;{(-1)^0 \; \binom{0}{0} \; Q'(\alpha)^{0} \; Q(\alpha)^0}=\frac{A}{Q(\alpha)},\\
c_{n-1} &=&  (-1) \; \frac{n!}{(n-1)!} \; \frac{A}{Q(\alpha)^{2}} \;{(-1)^0 \; \binom{1}{0} \; Q'(\alpha)^{1} \; Q(\alpha)^0} =  - n A \frac{Q'(\alpha)}{Q(\alpha)^2},
\end{eqnarray*}
which means that the base cases are verified.
For the inductive step, we assume that, for $i=0, 1, 2, ..., n-2$, $c_{i+1}$ and $c_{i+2}$ are given by \eqref{ciNoRoot}, and we want to prove that $c_i$ is also given by \eqref{ciNoRoot}.

From Theorem \ref{Teo1}, we know that $c_i=-\frac{i+1}{Q(\alpha)} \prect{Q'(\alpha)  c_{i+1} + (i+2) c_{i+2}}$ for $i=0, 1, 2, ..., n-2$. Plugging the expressions of $c_{i+1}$ and $c_{i+2}$, which are defined by \eqref{ciNoRoot},  in the expression of $c_i$ we obtain
\begin{eqnarray*}
-\frac{i+1}{Q(\alpha)} && \prectI{Q'(\alpha)(-1)^{n-i-1} \frac{n!}{(i+1)!}  \frac{A}{Q(\alpha)^{n-i}} \times}  \\
 && \sum\limits_{\substack{j=0 \\ j \in \N_0}}^{\frac{n-i}{2}-\frac{1}{2}}{(-1)^j  \binom{n-i-j-1}{j}  Q'(\alpha)^{n-i-2j-1}  Q(\alpha)^j} \\
&& +(i+2)(-1)^{n-i-2}  \frac{n!}{(i+2)!} \frac{A}{Q(\alpha)^{n-i-1}} \times \\
&& \prectF{ \sum\limits_{\substack{j=0 \\ j \in \N_0}}^{\frac{n-i}{2}-1}{(-1)^j  \binom{n-i-j-2}{j}  Q'(\alpha)^{n-i-2j-2} Q(\alpha)^j}}. 
\end{eqnarray*}
Rearranging the terms and changing the variable in the second sum, we get
\begin{eqnarray*}
 (-1)^{n-i} \frac{n!}{i!} \frac{A}{Q(\alpha)^{n-i+1}} && \prectI{\; \sum\limits_{\substack{j=0 \\ j \in \N_0}}^{\frac{n-i}{2}-\frac{1}{2}}{(-1)^j \; \binom{n-i-j-1}{j} \; Q'(\alpha)^{n-i-2j} \; Q(\alpha)^j}}\\
&& + \prectF{\sum\limits_{\substack{j=1 \\ j \in \N_0}}^{\frac{n-i}{2}}{(-1)^{j} \; \binom{n-i-j-1}{j-1} \; Q'(\alpha)^{n-i-2j} \; Q(\alpha)^{j}}}.
\end{eqnarray*}
Joining the two sums and taking into account some permutation's properties, we end up with the following expression
\begin{eqnarray*}
(-1)^{n-i} \frac{n!}{i!} \frac{A}{Q(\alpha)^{n-i+1}} && \prectI{\; \sum\limits_{\substack{j=1 \\ j \in \N_0}}^{\frac{n-i}{2}-\frac{1}{2}}{(-1)^j \binom{n-i-j}{j} \; Q'(\alpha)^{n-i-2j} \; Q(\alpha)^j}}\\
&& + \prectF{Q'(\alpha)^{n-i} + (-1)^{\frac{n-i}{2}} \; Q(\alpha)^{\frac{n-i}{2}} \chi_{\chav{n-i \; \text{is even} }}}.
\end{eqnarray*}
Finally, we conclude that
\begin{eqnarray*}
c_i &=& (-1)^{n-i} \frac{n!}{i!} \frac{A}{Q(\alpha)^{n-i+1}} \sum\limits_{\substack{j=0 \\ j \in \N_0}}^{\frac{n-i}{2}}{(-1)^j \binom{n-i-j}{j} \; Q'(\alpha)^{n-i-2j} \; Q(\alpha)^j},
\end{eqnarray*}
which coincides with the expression given by \eqref{ciNoRoot}. Thus the proof for the first case is finished.

\item If $\alpha$ is a root of $Q$ with multiplicity one, as we proved before, the particular solution is of the form $y_p(x) = x^\alpha \sum_{i=0}^{n} c_i \paren{\ln x}^{i+1}$, where
$c_n=\frac{A}{(n+1) Q'(\alpha)}$ and $c_i = - \frac{i+2}{Q'(\alpha)} c_{i+1}$, for $i=0,1,2, ..., n-1$.
We want to prove that we can write the coefficients  $c_i$ in the general way presented  in \eqref{ci1Root}.

As before, we use backwards mathematical induction. Starting with the base case and
taking into account \eqref{ci1Root}, we have
$$c_n = (-1)^{0} \; \frac{n!}{(n+1)!} \; \frac{A}{Q'(\alpha)}=\frac{A}{(n+1) Q'(\alpha)},$$
which coincides with the expression given by Theorem \ref{Teo1}. Thus, the base case is verified.
To prove the induction step, for $i=0,1,2,..,n-1$, we assume that $c_{i+1}$ is given by \eqref{ci1Root} and we want to prove that $c_i$ is also given by \eqref{ci1Root}.

From Theorem \ref{Teo1}, we know that  $c_i =  - \frac{i+2}{Q'(\alpha)} c_{i+1}$, for $i=0,1,2, ..., n-1$. Plugging in $c_i$ the expression of $c_{i+1}$, which is given by \eqref{ci1Root}, we obtain
\begin{equation*}
c_i = - \frac{i+2}{Q'(\alpha)} \; (-1)^{n-i-1} \; \frac{n!}{(i+2)!} \; \frac{A}{Q'(\alpha)^{n-i}} = (-1)^{n-i} \; \frac{n!}{(i+1)!} \; \frac{A}{Q'(\alpha)^{n-i+1}},
\end{equation*}
and therefore the induction step is proved. With this we conclude the proof.
\end{enumerate}
\end{proof}

A special case of the previous theorem is when $n=0$. In this case the differential equation is
\begin{equation*}\label{eq}
x^2 y^{\prime \prime}(x) + a x y^{\prime}(x) + b y(x) = A x^\alpha.
\end{equation*}
Using the results proved before, the corresponding particular solution is given by
$$y_p(x) = \varphi \; x^\alpha \paren{\ln x}^r,$$ 
where $\varphi=\frac{A}{Q^{(r)} (\alpha)}$\footnote{$Q^{(r)}(\alpha)$ is the derivative of order $r$ of $Q$ w.r.t. $\alpha$. In particular, if $r=0$ we consider that $Q^{(r)}(\alpha)$ is exactly $Q(\alpha)$.}, with  $r$\footnote{$r$ can take the values $0,1$ or $2$. We consider $r=0$ when $\alpha$ is not a root of $Q$.} being the multiplicity of $\alpha$ as a root of $Q$.

Note that in Example \ref{Example1} we could use this result many times. Namely, to determine $V^1_p$ we could use it twice:
\begin{itemize}
\item  with $\alpha=\theta$ and $A = c \rho (1+ \kappa)^\theta$ obtaining the solution $\frac{c \rho (1+ \kappa)^\theta}{Q(\theta)} x^\theta$, once we assumed that $\theta$ is not a root of $Q$;
\item with $\alpha=0$ and $A = - c I$ obtaining the solution $- \frac{c I}{b}$.
\end{itemize}
%
%
%
%
%
%
Clearly, for the other branches, $V_2$ and $V_3$,  we could also use this special case in some situations. For others, we would need to use the general expression presented in the theorem. However, given the type of expressions of $V_p^i$ (see \eqref{eq:Vpi}), in general, we need to find a solution for a sum of functions.
Taking this into account, we generalize the  Theorem \ref{Teo2} in the following corollary.

\begin{corollary}
Consider the following second order ODE:
\begin{equation}\label{generalCase}
x^2 y^{\prime \prime}(x) + a x y^{\prime}(x) + b y(x) = \sum_{k=1}^{m} A_k x^{\alpha_k} \paren{\ln x}^{n_k},
\end{equation}
with $x>0$, $a, b \in \R$, $\alpha_k, A_k \in \R\setminus\chav{0}$ and $n_k \in \N_0$, for $k=1,2,...,m$, with $m \in \N$.
Then the particular solution of \eqref{generalCase} is of the form $y_p(x) = \sum_{k=1}^m y_{p_k}(x)$\footnote{Note that $y_p$ has at least $m$ parcels and at most $m+\sum_{k=1}^m n_k$ parcels. 
When $\alpha_1=\alpha_2=...=\alpha_m$ are roots of $Q$ all with multiplicity two, $y_p$ has $m$ parcels. 
Oppositely, when none of the $\alpha_k$ (with $k=1,2,...,m$) has multiplicity two, $y_p$ has $m+\sum_{k=1}^m n_k$ parcels.}, where $y_{p_k}(x)$ is the solution of the equation
\begin{equation*}
x^2 y_k^{\prime \prime}(x) + a x y_k^{\prime}(x) + b y_k(x) =  A_k x^{\alpha_k} \paren{\ln x}^{n_k},
\end{equation*}
which is presented in Theorem \ref{Teo2}.
\end{corollary}

\section{CONCLUSIONS}\label{C}
In this paper we present a contribution to optimal stopping problems when the decision to stop may be taken due to a jump that leads directly to the stopping region, across the boundary state.

The equation that holds in the continuation region is a differential--difference equation, whose solution we are able to provide, using a backward argument.
Firstly, we get a recursive relation of the involved coefficients, which is then used to derive a closed form expression for the particular solution.

Using these results, we may, in specific situations, derive the value function, which will be a piecewise function, with an infinite number of subdomains, as we present in \eqref{Vdefinition}.

\begin{acknowledgments}
Most of Rita Pimentel's research was supported by her Ph.D. research grant from FCT with reference number SFRH/BD/97259/2013. The final part of the work was carried out during the tenure of an ERCIM `Alain Bensoussan' Fellowship Programme.\\
The authors also thank to Carlos Oliveira for the discussions and all valuable suggestions.
\end{acknowledgments}



\end{document}